\documentclass[a4paper,10pt,reqno]{amsart}
\usepackage[utf8x]{inputenc}

\usepackage{amsmath,amssymb,amsthm,amsxtra,mathrsfs,bm,bbm}
\usepackage{graphicx}
\usepackage[longnamesfirst,authoryear]{natbib}
\usepackage{url}
\usepackage[colorlinks,citecolor=blue]{hyperref}

\newcommand{\ee}{\mathrm{e}}
\newcommand{\ii}{\mathrm{i}}

\newcommand{\R}{\mathbb{R}}

\newcommand{\uu}{\boldsymbol{u}}
\newcommand{\vv}{\boldsymbol{v}}
\newcommand{\ww}{\boldsymbol{w}}

\newcommand{\BB}{\boldsymbol{B}}

\newcommand{\Ss}{\mathcal{S}}

\newcommand{\eps}{\varepsilon}

\newcommand{\weakstarto}{\overset{*}{\rightharpoonup}}

\newcommand{\pd}{\partial}
\newcommand{\rd}{\mathrm{d}}

\newcommand{\Grad}{\nabla}
\newcommand{\Div}{\nabla \cdot}
\newcommand{\Curl}{\nabla \times}
\newcommand{\Laplace}{\Delta}

\newcommand{\abs}[1]{\left| #1 \right|}
\newcommand{\norm}[1]{\| #1 \|}
\newcommand{\bignorm}[1]{\left\| #1 \right\|}

\newcommand{\normBig}[1]{\Big\| #1 \Big\|}

\newcommand{\inner}[2]{\langle #1 , #2 \rangle}

\theoremstyle{plain}
\newtheorem{thm}{Theorem}[section]
\newtheorem{prop}[thm]{Proposition}
\newtheorem{lemma}[thm]{Lemma}

\newtheorem{cor}[thm]{Corollary}

\numberwithin{equation}{section}

\makeatletter
\newtheorem*{rep@theorem}{\rep@title}
\newcommand{\newreptheorem}[2]{\newenvironment{rep#1}[1]{\def\rep@title{#2 \ref{##1}}\begin{rep@theorem}}{\end{rep@theorem}}}
\makeatother

\newreptheorem{thm}{Theorem}

\title[Commutator estimates and local existence for non-resistive MHD]{Higher order commutator estimates and local existence for the non-resistive MHD equations and related models}

\author[C.\ L.\ Fefferman]{Charles L.\ Fefferman}
\thanks{CLF is supported by NSF grant DMS-09-0104. DSMcC is a member of the Warwick ``MASDOC'' doctoral training centre, which is funded by EPSRC grant EP/HO23364/1. JCR is supported by an EPSRC Leadership Fellowship EP/G007470/1.}
\address{C.\ L.\ Fefferman \\
Department of Mathematics \\
Princeton University \\
Fine Hall \\
Washington Road \\
Princeton, NJ 08544}
\email{cf@math.princeton.edu}

\author[D.\ S.\ McCormick]{David S.\ McCormick}
\address{D.\ S.\ McCormick \\
Mathematics Institute \\
University of Warwick \\
Coventry, CV4 7AL \\
United Kingdom}
\email{d.s.mccormick@warwick.ac.uk}

\author[J.\ C.\ Robinson]{James C.\ Robinson}
\address{J.\ C.\ Robinson \\
Mathematics Institute \\
University of Warwick \\
Coventry, CV4 7AL \\
United Kingdom}
\email{j.c.robinson@warwick.ac.uk}

\author[J.\ L.\ Rodrigo]{Jose L.\ Rodrigo}
\address{J.\ L.\ Rodrigo \\
Mathematics Institute \\
University of Warwick \\
Coventry, CV4 7AL \\
United Kingdom}
\email{j.l.rodrigo@warwick.ac.uk} 

\date{January 20, 2014}



\makeatletter
\hypersetup{
	pdftitle = {\@title},
	pdfauthor = {\shortauthors},
	pdfsubject = {2010 MSC: \@subjclass},
	pdfkeywords = {\@keywords},
}
\makeatother

\begin{document}

\maketitle

\begin{abstract}
This paper establishes the local-in-time existence and uniqueness of strong solutions in $H^{s}$ for $s > n/2$ to the viscous, non-resistive magnetohydrodynamics (MHD) equations in $\R^{n}$, $n=2, 3$, as well as for a related model where the advection terms are removed from the velocity equation. The uniform bounds required for proving existence are established by means of a new estimate, which is a partial generalisation of the commutator estimate of Kato \& Ponce (Comm. Pure Appl. Math. \textbf{41}(7), 891--907, 1988).
\end{abstract}

\section{Introduction}

In this paper we prove local-in-time existence of strong solutions to the non-resistive magnetohydrodynamics (MHD) equations:
\begin{subequations}
\label{eqn:MHD}
\begin{align}
\frac{\pd \uu}{\pd t} + (\uu \cdot \Grad) \uu - \nu \Laplace \uu + \Grad p_{*} &= (\BB \cdot \Grad) \BB, \label{eqn:MHD-u} \\
\frac{\pd \BB}{\pd t} + (\uu \cdot \Grad) \BB &= (\BB \cdot \Grad) \uu, \label{eqn:MHD-B} \\
\Div \uu = \Div \BB &= 0 \label{eqn:MHD-Div}
\end{align}
\end{subequations}
on the whole of $\R^{n}$ with $n=2, 3$, with divergence-free initial data $\uu_{0}, \BB_{0} \in H^{s}(\R^{n})$, for $s > n/2$. In particular, we prove the following theorem.

\begin{thm}
\label{thm:MHDLocalExistence}
For $s > n/2$, and initial data $\uu_{0}, \BB_{0} \in H^{s}(\R^{n})$ with $\Div \uu_{0} = \Div \BB_{0} = 0$, there exists a time $T_{*} = T_{*}(s, \norm{\uu_{0}}_{H^{s}}, \norm{\BB_{0}}_{H^{s}}) > 0$ such that the equations \eqref{eqn:MHD} have a unique solution $(\uu, \BB)$, with $\uu, \BB \in C([0, T_{*}]; H^{s}(\R^{n}))$.
\end{thm}

Note that there is no diffusion term in \eqref{eqn:MHD-B}. When this term ($-\eta \Laplace \BB$) is also present, in 2D one has global existence and uniqueness of weak solutions, and in 3D one has local existence of weak solutions, much like the Navier--Stokes equations; these results go back to \cite{art:DuvautLions1972} and \cite{art:SermangeTemam1983}.

By contrast, for the system \eqref{eqn:MHD} with diffusion only in \eqref{eqn:MHD-u}, \cite{art:JiuNiu2006} established local existence of solutions in 2D for initial data in $H^{s}$, but only for integer $s \geq 3$. They also proved a conditional regularity result in 2D: the solution to \eqref{eqn:MHD} can be extended beyond time $T$ if $\BB \in L^{p}(0, T; W^{2,q}(\R^{2}))$, for $\frac{2}{p} + \frac{1}{q} \leq 2$, and $1 \leq p \leq \frac{4}{3}$, $2 < q \leq \infty$. This was generalised by \cite{art:ZhouFan2011}, who showed that $\Grad \BB \in L^{1}(0, T; \mathrm{BMO}(\R^{2}))$ suffices. In 3D, \cite{art:FanOzawa2009} established a similar conditional regularity result, showing that the solution can be extended beyond time $T$ if $\Grad \uu \in L^{1}(0, T; L^{\infty}(\R^{3}))$.

Intriguingly, with diffusion for $\BB$ but \emph{not} for $\uu$, \cite{art:Kozono1989} proved global existence of weak solutions in 2D for divergence-free initial data in $L^{2}$; while in 3D, \cite{art:FanOzawa2009} showed that, again, the solution can be extended beyond time $T$ if $\Grad \uu \in L^{1}(0, T; L^{\infty}(\R^{3}))$.

In the ideal case, with no diffusion in either equation, \cite{art:Schmidt1988} and \cite{art:Secchi1993} established local existence of strong solutions when the initial data is in $H^{s}$ for integer $s > 1 + n/2$, while \cite{art:CKS1997} proved a conditional regularity result for fully ideal MHD which corresponds to the conditional regularity result for Euler due to \cite{art:BKM1984}: namely, if
\[
\int_{0}^{T} \left( \norm{\Curl \uu(\tau)}_{\infty} + \norm{\Curl \BB(\tau)}_{\infty} \right) \, \rd \tau < \infty,
\]
then the solution can be continued beyond time $T$.

The system \eqref{eqn:MHD} is connected with the method of \emph{magnetic relaxation}, an idea discussed by \cite{art:Moffatt1985}. Formally, we obtain the standard energy estimate
\[
\frac{1}{2} \frac{\rd}{\rd t} \left( \norm{\uu}_{L^{2}}^{2} + \norm{\BB}_{L^{2}}^{2} \right) + \nu \norm{\Grad \uu}_{L^{2}}^{2} = 0;
\]
therefore, as long as $\uu$ is not identically zero, the energy should decay. Thus, the magnetic forces on a viscous non-resistive plasma should come to equilibrium, so that the fluid velocity $\uu$ tends to zero. We should be left with a steady magnetic field $\BB$ that satisfies $(\BB \cdot \Grad) \BB - \Grad p_{*} = 0$, which up to a change of sign for the pressure are the stationary Euler equations.

However, while this is a useful heuristic argument there is as yet no rigorous proof that the method should yield a stationary Euler flow, not least because there is no global existence result for the system \eqref{eqn:MHD}, even in 2D. Nonetheless, \cite{art:Nunez2007} proved that $\norm{\uu(t)}_{L^{2}} \to 0$ as $t \to \infty$, if we assume a smooth solution to \eqref{eqn:MHD} exists for all time, and that the solution satisfies $\norm{\BB(t)}_{L^{\infty}} \leq M$ for all $t$. We should note that \cite{art:EncisoPeraltaSalas2012} proved the existence of a stationary Euler flow, albeit with infinite energy, with stream or vortex lines of prescribed link type; but whether such flows arise as limits of system \eqref{eqn:MHD} is still very much open.

The main difficulty in proving local existence (Theorem~\ref{thm:MHDLocalExistence}) with diffusion only in the $\uu$ equation stems from the nonlinear terms. Naively, $H^{s}$ is an algebra for $s > n/2$, so one obtains
\[
\abs{\inner{(\uu \cdot \Grad)\vv}{\ww}_{H^{s}}} \leq \norm{\uu}_{H^{s}} \norm{\Grad \vv}_{H^{s}} \norm{\ww}_{H^{s}}.
\]
For three of the four nonlinear terms, this is sufficient, but for the $(\uu \cdot \Grad) \BB$ term we must estimate $\norm{\Grad \BB}_{H^{s}}$, and if we start with $\BB_{0} \in H^{s}$ we have no control over the $H^{s}$ norm of $\Grad \BB$ because there is no smoothing for $\BB$.

We will show that for $s > n/2$ one can in fact obtain the bound
\[
\abs{\inner{(\uu \cdot \Grad)\BB}{\BB}_{H^{s}}} \leq c \norm{\Grad \uu}_{H^{s}} \norm{\BB}_{H^{s}}^{2}.
\]
This is a consequence of a new commutator estimate applicable to the nonlinear terms. To describe this, let $J^{s}$ and $\Lambda^{s}$ denote fractional derivative operators defined in terms of Fourier transforms\footnote{Note that we use the definition $\mathscr{F}[f](\xi) = \hat{f}(\xi) = \int_{\R^{n}} \ee^{-2\pi \ii x \cdot \xi} f(x) \, \rd x$.} as follows:
\[
\mathscr{F}[ J^{s} f ] (\xi) = (1 + |\xi|^{2})^{s/2} \hat{f}(\xi), \qquad \mathscr{F}[ \Lambda^{s} f ] (\xi) = |\xi|^{s} \hat{f}(\xi).
\]
It was proved in \cite{art:KatoPonce1988} that, for $s \geq 0$ and $1 < p < \infty$, the nonlinear terms satisfy the following estimate:
\[
\norm{J^{s} [ (\uu \cdot \Grad) \BB ] - (\uu \cdot \Grad) (J^{s} \BB) }_{L^{p}} \leq c ( \norm{\Grad \uu}_{L^{\infty}} \norm{J^{s-1} \Grad \BB}_{L^{p}} + \norm{J^{s} \uu}_{L^{p}} \norm{\Grad \BB}_{L^{\infty}} )
\]
which, for $p = 2$ and $s > n/2$, implies the following:
\begin{equation}
\label{eqn:KatoPonce}
\norm{J^{s} [ (\uu \cdot \Grad) \BB ] - (\uu \cdot \Grad) (J^{s} \BB) }_{L^{2}} \leq c ( \norm{\Grad \uu}_{H^{s}} \norm{\BB}_{H^{s}} + \norm{\uu}_{H^{s}} \norm{\Grad \BB}_{H^{s}} ).
\end{equation}
Once again, however, estimate \eqref{eqn:KatoPonce} cannot immediately be applied to our system of equations, because the second term on the right-hand side still contains $\norm{\Grad \BB}_{H^{s}}$; we thus require, and now prove, a similar estimate that only contains the first of the two terms on the right-hand side of \eqref{eqn:KatoPonce}.

\begin{thm}
\label{thm:Commutator}
Given $s > n/2$, there is a constant $c = c(n,s)$ such that, for all $\uu, \BB$ with $\Grad \uu, \BB \in H^{s}(\R^{n})$,
\begin{equation}
\label{eqn:Commutator}
\norm{ \Lambda^{s} [(\uu \cdot \Grad) \BB] - (\uu \cdot \Grad) (\Lambda^{s}\BB) }_{L^{2}} \leq c \norm{\Grad \uu}_{H^{s}} \norm{\BB}_{H^{s}}.
\end{equation}
\end{thm}

The rest of the paper is structured as follows. In Section~\ref{sec:Commutator}, we prove Theorem~\ref{thm:Commutator}. In Section~\ref{sec:LocalExistenceMHD}, we use Theorem~\ref{thm:Commutator} and various other standard techniques to prove Theorem~\ref{thm:MHDLocalExistence}. In Section~\ref{sec:LocalExistenceCEP}, we outline a proof of local existence for a related model, namely equations~\eqref{eqn:MHD} with the $\frac{\pd \uu}{\pd t} + (\uu \cdot \Grad) \uu$ terms removed from the first equation, which we previously studied in \cite{art:ARMA}. Finally, in Appendix~\ref{app:Counterexample}, we exhibit a counterexample to show that our commutator estimate does not hold in the case $s = n/2$, at least for $n=2$, even if $\uu$ and $\BB$ are required to be divergence-free; this therefore suggests that proving local existence in $H^{n/2}$ (if possible) would require a more refined technique. We note that in a recent paper \cite{art:BourgainLi2013} showed that the Euler equations on $\R^{n}$ are in fact ill-posed in $H^{1+n/2}$ ($n=2,3$); in light of this it seems likely that system~\eqref{eqn:MHD} is ill-posed in $H^{n/2}$.

\section{Commutator estimates}
\label{sec:Commutator}

In this section, we prove the following commutator estimate.

\begin{repthm}{thm:Commutator}
Given $s > n/2$, there is a constant $c = c(n,s)$ such that, for all $\uu, \BB$ with $\Grad \uu, \BB \in H^{s}(\R^{n})$,
\begin{equation}
\norm{ \Lambda^{s} [(\uu \cdot \Grad) \BB] - (\uu \cdot \Grad) (\Lambda^{s}\BB) }_{L^{2}} \leq c \norm{\Grad \uu}_{H^{s}} \norm{\BB}_{H^{s}}. \tag{\ref{eqn:Commutator}}
\end{equation}
\end{repthm}

Before embarking on the proof, we note that a priori the left-hand side makes sense only when $\uu, \Grad \BB \in H^{s}(\R^{n})$; however, the right-hand side is finite when $\Grad \uu, \BB \in H^{s}(\R^{n})$, and since both sides are linear in $\uu$ and $\BB$ it suffices to prove the inequality for $\uu, \BB \in C_{c}^{\infty}(\R^{n})$ and use the density of $C_{c}^{\infty}(\R^{n})$ in $H^{s}(\R^{n})$.

\begin{proof}
Let $\uu, \BB \in C_{c}^{\infty}(\R^{n})$. First, note that
\[
\mathscr{F} [ (\uu \cdot \Grad) \BB_{k} ] (\xi) = \sum_{j=1}^{n} \widehat{(\uu_{j} \pd_{j} \BB_{k})} (\xi) = \sum_{j=1}^{n} \int \hat{\uu}_{j}(\zeta) (\xi - \zeta)_{j} \hat{\BB}_{k}(\xi - \zeta) \, \rd \zeta,
\]
so
\[
\mathscr{F} [ \Lambda^{s}[(\uu \cdot \Grad) \BB_{k}] ] (\xi) = |\xi|^{s} \sum_{j=1}^{n} \int \hat{\uu}_{j}(\zeta) (\xi - \zeta)_{j} \hat{\BB}_{k}(\xi - \zeta) \, \rd \zeta.
\]
Similarly,
\[
\mathscr{F} [ (\uu \cdot \Grad) (\Lambda^{s} \BB_{k}) ] (\xi) = \sum_{j=1}^{n} \int \hat{\uu}_{j}(\zeta) (\xi - \zeta)_{j} |\xi - \zeta|^{s} \hat{\BB}_{k}(\xi - \zeta) \, \rd \zeta.
\]
Therefore the Fourier transform of $\Lambda^{s} [(\uu \cdot \Grad) \BB] - (\uu \cdot \Grad) (\Lambda^{s}\BB)$ is
\[
\sum_{j=1}^{n} \int (|\xi|^{s} - |\xi - \zeta|^{s}) \hat{\uu}_{j}(\zeta) (\xi - \zeta)_{j} \hat{\BB}_{k}(\xi - \zeta) \, \rd \zeta;
\]
by Parseval's identity it suffices to bound this in $L^{2}$.

We split the integral into the two regions $|\zeta| < |\xi|/2$ and $|\zeta| \geq |\xi|/2$. In the first region $|\zeta| < |\xi|/2$, we use the inequality
\begin{equation}
\label{eqn:GradientEstimate}
\abs{ |\xi|^{s} - |\xi - \zeta|^{s} } \leq c |\xi - \zeta|^{s-1} |\zeta|,
\end{equation}
whose proof we postpone, to obtain
\begin{align*}
&\sum_{j=1}^{n} \int_{|\zeta| < |\xi|/2} (|\xi|^{s} - |\xi - \zeta|^{s}) \hat{\uu}_{j}(\zeta) (\xi - \zeta)_{j} \hat{\BB}_{k}(\xi - \zeta) \, \rd \zeta \\
&\qquad \qquad \qquad \leq c \int |\zeta| |\hat{\uu}(\zeta)|  |\xi - \zeta|^{s} |\hat{\BB}(\xi - \zeta)| \, \rd \zeta.
\end{align*}
By Young's inequality, the $L^{2}$ norm of the above integral expression is bounded above by
\begin{align*}
&\normBig{|\zeta| |\hat{\uu}(\zeta)|}_{L^{1}} \normBig{|\eta|^{s} |\hat{\BB}(\eta)|}_{L^{2}} \\
&\qquad \leq \bignorm{\frac{1}{(1 + |\zeta|^{2})^{s/2}}}_{L^{2}} \normBig{(1 + |\zeta|^{2})^{s/2} |\zeta| |\hat{\uu}(\zeta)|}_{L^{2}} \normBig{|\eta|^{s} |\hat{\BB}(\eta)|}_{L^{2}} \\
&\qquad \leq c \norm{\Grad \uu}_{H^{s}} \norm{\BB}_{\dot{H}^{s}},
\end{align*}
since $(1 + |\zeta|^{2})^{-s/2} \in L^{2}$ as $s > n/2$.

In the second region $|\zeta| \geq |\xi|/2$, we have $|\xi| \leq 2 |\zeta|$ and $|\xi - \zeta| \leq 3|\zeta|$. So
\[
\abs{ |\xi|^{s} - |\xi - \zeta|^{s} } \leq c |\zeta|^{s},
\]
hence
\begin{align*}
&\sum_{j=1}^{n} \int_{|\zeta| \geq |\xi|/2} (|\xi|^{s} - |\xi - \zeta|^{s}) \hat{\uu}_{j}(\zeta) (\xi - \zeta)_{j} \hat{\BB}_{k}(\xi - \zeta) \, \rd \zeta \\
&\qquad \qquad \qquad \leq c \int |\zeta|^{s+1} |\hat{\uu}(\zeta)| |\hat{\BB}(\xi - \zeta)| \, \rd \zeta.
\end{align*}
The $L^{2}$ norm of the above integral expression is bounded by
\begin{align*}
&\normBig{|\zeta|^{s+1} |\hat{\uu}(\zeta)|}_{L^{2}} \normBig{|\hat{\BB}(\eta)|}_{L^{1}} \\
&\qquad \leq \normBig{|\zeta|^{s+1} |\hat{\uu}(\zeta)|}_{L^{2}} \bignorm{\frac{1}{(1 + |\eta|^{2})^{s/2}}}_{L^{2}} \normBig{(1 + |\eta|^{2})^{s/2} |\hat{\BB}(\eta)|}_{L^{2}} \\
&\qquad \leq c \norm{\Grad \uu}_{\dot{H}^{s}} \norm{\BB}_{H^{s}},
\end{align*}
since $s > n/2$. This completes the proof when $\uu, \BB \in C_{c}^{\infty}(\R^{n})$, and the general case follows by density of $C_{c}^{\infty}(\R^{n})$ in $H^{s}(\R^{n})$.

It remains to prove inequality \eqref{eqn:GradientEstimate}: given $\xi$ and $\zeta$, let $h(t) = |\xi - t\zeta|^{s}$. As $|\zeta| < |\xi|/2$, $h$ is smooth on $[0,1]$. Now
\[
h'(t) = -s |\xi - t \zeta|^{s-2} (\xi - t \zeta) \cdot \zeta,
\]
so applying the mean value theorem to $h$ on $[0,1]$ we obtain
\[
\abs{ |\xi|^{s} - |\xi - \zeta|^{s} } \leq \max_{t \in [0,1]} |h'(t)| \leq s |\zeta| \max_{t \in [0,1]} |\xi - t \zeta|^{s-1}.
\]
As $|\zeta| < |\xi|/2$, for all $t \in [0,1]$,
\[
\frac{|\xi|}{2} \leq |\xi - t \zeta| \leq \frac{3|\xi|}{2};
\]
in particular $\frac{|\xi|}{2} \leq |\xi - \zeta|$ and so
\[
|\xi - t \zeta| \leq \frac{3|\xi|}{2} \leq 3|\xi - \zeta|,
\]
whence \eqref{eqn:GradientEstimate} follows.
\end{proof}

Before proceeding, let us note that the result of Theorem~\ref{thm:Commutator} cannot be extended to the case $s=1$ when $n=2$: in Appendix~\ref{app:Counterexample} we give an example to show that the inequality
\begin{equation}
\label{eqn:False2DIneq}
\norm{\pd_{k} [ (\uu \cdot \Grad) \BB] - (\uu \cdot \Grad) (\pd_{k} \BB)}_{L^{2}} \leq c \norm{\Grad \uu}_{H^{1}} \norm{\BB}_{H^{1}}
\end{equation}
cannot hold in dimension $2$, by exhibiting a pair of divergence-free functions $\uu$ and $\BB$ for which the right-hand side is finite, but the left-hand side is infinite. As a result, it is clear that if we were to try to prove local existence with initial data in $H^{s}$ for $s = n/2$ then a different approach would be required.

Using the fact that, when $\uu$ is divergence-free,
\[
\inner{(\uu \cdot \Grad) (\Lambda^{s}\BB)}{\Lambda^{s}\BB} = 0,
\]
we immediately obtain the following corollary of Theorem~\ref{thm:Commutator}.

\begin{cor}
\label{cor:Commutator}
Given $s > n/2$, there is a constant $c = c(n,s)$ such that, for all $\uu, \BB$ with $\Grad \uu, \BB \in H^{s}(\R^{n})$ and $\Div \uu = 0$,
\[
\abs{\inner{\Lambda^{s} [(\uu \cdot \Grad) \BB]}{\Lambda^{s} \BB}} \leq c \norm{\Grad \uu}_{H^{s}} \norm{\BB}_{H^{s}}^{2}.
\]
\end{cor}

We will use Corollary~\ref{cor:Commutator} in the next section to prove local existence of solutions to equations~\eqref{eqn:MHD} with initial data in $H^{s}$ for $s > n/2$.

\section{Local existence for viscous non-resistive magnetohydrodynamics}
\label{sec:LocalExistenceMHD}

We return to the equations
\begin{subequations}
\begin{align}
\frac{\pd \uu}{\pd t} + (\uu \cdot \Grad) \uu - \nu \Laplace \uu + \Grad p_{*} &= (\BB \cdot \Grad) \BB, \tag{\ref{eqn:MHD-u}} \\
\frac{\pd \BB}{\pd t} + (\uu \cdot \Grad) \BB &= (\BB \cdot \Grad) \uu, \tag{\ref{eqn:MHD-B}} \\
\Div \uu = \Div \BB &= 0 \tag{\ref{eqn:MHD-Div}}
\end{align}
\end{subequations}
on the whole of $\R^{n}$, with initial data $\uu_{0}, \BB_{0} \in H^{s}(\R^{n})$ satisfying $\Div \uu_{0} = \Div \BB_{0} = 0$, for $s > n/2$. We will prove the following theorem.

\addtocounter{equation}{-1}

\begin{repthm}{thm:MHDLocalExistence}
For $s > n/2$, and initial data $\uu_{0}, \BB_{0} \in H^{s}(\R^{n})$ with $\Div \uu_{0} = \Div \BB_{0} = 0$, there exists a time $T_{*} = T_{*}(s, \norm{\uu_{0}}_{H^{s}}, \norm{\BB_{0}}_{H^{s}}) > 0$ such that the equations \eqref{eqn:MHD} have a unique solution $(\uu, \BB)$, with $\uu, \BB \in C([0, T_{*}]; H^{s}(\R^{n}))$.
\end{repthm}

The general strategy of the proof is similar to that for proving existence of solutions to the Navier--Stokes and Euler equations which can be found in Section~3.2 of \cite{book:MajdaBertozzi}, for example. First, we show that the solutions $(\uu^{R}, \BB^{R})$ of some smoothed version of the equations exist and are uniformly bounded in $H^{s}$. We then show they are Cauchy in the $L^{2}$ norm as $R \to \infty$. By interpolation, $(\uu^{R}, \BB^{R}) \to (\uu, \BB)$ in any $H^{s'}$ for $0 < s' < s$, which implies that $(\uu, \BB)$ solve the original equations.

Define the Fourier truncation $\Ss_{R}$ as follows:
\[
\widehat{\Ss_{R} f}(\xi) = \mathbbm{1}_{B_{R}}(\xi) \hat{f}(\xi),
\]
where $B_{R}$ denotes the ball of radius $R$ centered at the origin. Note that
\begin{align*}
\norm{\Ss_{R}f - f}_{H^{s}}^{2} &= \int_{(B_{R})^{c}} (1 + |\xi|^{2})^{s} |\hat{f}(\xi)|^{2} \, \rd \xi \\
&= \int_{(B_{R})^{c}} \frac{1}{(1 + |\xi|^{2})^{k}} (1 + |\xi|^{2})^{s+k} |\hat{f}(\xi)|^{2} \, \rd \xi \\
&\leq \frac{1}{(1 + R^{2})^{k}} \int_{(B_{R})^{c}} (1 + |\xi|^{2})^{s+k} |\hat{f}(\xi)|^{2} \, \rd \xi \\
&\leq \frac{C}{R^{2k}} \norm{f}_{H^{s+k}}^{2}.
\end{align*}
Hence
\begin{align}
\norm{\Ss_{R} f - f}_{H^{s}} &\leq C (1/R)^{k} \norm{f}_{H^{s+k}}, \label{eqn:MollifierProp1}\\
\norm{\Ss_{R} f - \Ss_{R'} f}_{H^{s}} &\leq C \max\{(1/R)^{k}, (1/R')^{k}\} \norm{f}_{H^{s+k}}. \label{eqn:MollifierProp2}
\end{align}

We consider the truncated MHD equations on the whole of $\R^{n}$:
\begin{subequations}
\label{eqn:MHD-Mollified}
\begin{align}
\frac{\pd \uu^{R}}{\pd t} - \nu \Laplace \uu^{R} +  \Grad p_{*}^{R} &= \Ss_{R}[(\BB^{R} \cdot \Grad) \BB^{R}] - \Ss_{R}[(\uu^{R} \cdot \Grad) \uu^{R}], \label{eqn:MHD-Mollified-u} \\
\frac{\pd \BB^{R}}{\pd t} &= \Ss_{R}[(\BB^{R} \cdot \Grad) \uu^{R}] - \Ss_{R}[(\uu^{R} \cdot \Grad) \BB^{R}], \label{eqn:MHD-Mollified-B} \\
\Div \uu^{R} &= \Div \BB^{R} = 0,
\end{align}
\end{subequations}
with initial data $\Ss_{R} \uu_{0}, \Ss_{R} \BB_{0}$. By taking the cutoff initial data as we have, we ensure that $\uu^{R}, \BB^{R}$ lie in the space
\[
V_{R} := \{ f \in L^{2}(\R^{n}) : \hat{f} \text{ is supported in } B_{R} \},
\]
as the truncations are invariant under the flow of the equations. The Fourier cutoffs act like mollifiers, smoothing the equation; in particular, on the space $V_{R}$ it is easy to show that
\[
F(\uu^{R}, \BB^{R}) := \Ss_{R}[(\uu^{R} \cdot \Grad) \BB^{R}]
\]
is Lipschitz in $\uu^{R}$ and $\BB^{R}$. Hence, by Picard's theorem for infinite-dimensional ODEs (see Theorem~3.1 in \cite{book:MajdaBertozzi}, for example), there exists a solution $(\uu^{R}, \BB^{R})$ in $V_{R}$ to \eqref{eqn:MHD-Mollified} for some  time interval $[0, T(R)]$. The solution will exist as long as $\norm{\uu^{R}}_{H^{s}}$ and $\norm{\BB^{R}}_{H^{s}}$ remain finite.

\begin{prop}
\label{prop:Uniform}
Given initial data $\uu_{0}, \BB_{0} \in H^{s}(\R^{n})$ with $s > n/2$, there exists a time $T_{*}$ such that the quantities
\[
\sup_{t \in [0,T_{*}]} \norm{\uu^{R}(t)}_{H^{s}}\text{, }\sup_{t \in [0,T_{*}]} \norm{\BB^{R}(t)}_{H^{s}} \text{, } \int_{0}^{T_{*}} \norm{\Grad \uu^{R} (t)}_{H^{s}}^{2} \, \rd t
\]
are bounded uniformly in $R$.
\end{prop}

Before embarking on the proof, we first prove a simple energy estimate: take the inner product of \eqref{eqn:MHD-Mollified-u} with  $\uu^{R}$ and the inner product of \eqref{eqn:MHD-Mollified-B} with $\BB^{R}$, and add to obtain
\begin{equation}
\label{eqn:EnergyEquality}
\frac{1}{2} \frac{\rd}{\rd t} ( \norm{\uu^{R}}_{L^{2}}^{2} + \norm{\BB^{R}}_{L^{2}}^{2}) + \nu \norm{\Grad \uu^{R}}_{L^{2}}^{2} = 0;
\end{equation}
integrating and using the fact that $\norm{\uu^{R}(0)}_{L^{2}} \leq \norm{\uu_{0}}_{L^{2}}$ and $\norm{\BB^{R}(0)}_{L^{2}} \leq \norm{\BB_{0}}_{L^{2}}$ yields
\begin{equation}
\label{eqn:EnergyEstimate}
\norm{\uu^{R}(t)}_{L^{2}}^{2} + \norm{\BB^{R}(t)}_{L^{2}}^{2} + 2\nu \int_{0}^{t} \norm{\Grad \uu^{R}(s)}_{L^{2}}^{2} \, \rd s \leq \norm{\uu_{0}}_{L^{2}}^{2} + \norm{\BB_{0}}_{L^{2}}^{2}. 
\end{equation}

\begin{proof}[Proof of Proposition~\ref{prop:Uniform}]
For $s > n/2$, apply $\Lambda^{s}$ to both equations:
\begin{align*}
\frac{\pd}{\pd t} \Lambda^{s} \uu^{R} - \nu \Laplace \Lambda^{s} \uu^{R} + \Grad \Lambda^{s} p_{*}^{R} &= \Ss_{R} \Lambda^{s} [(\BB^{R} \cdot \Grad) \BB^{R}]  - \Ss_{R} \Lambda^{s} [(\uu^{R} \cdot \Grad) \uu^{R}],\\
\frac{\pd}{\pd t} \Lambda^{s} \BB^{R} &= \Ss_{R} \Lambda^{s} [(\BB^{R} \cdot \Grad) \uu^{R}]  - \Ss_{R} \Lambda^{s} [(\uu^{R} \cdot \Grad) \BB^{R}].
\end{align*}
Take the inner product of the first equation with $\Lambda^{s} \uu^{R}$, and the inner product of the second equation with $\Lambda^{s} \BB^{R}$, to obtain
\begin{align*}
\frac{1}{2} \frac{\rd}{\rd t} \norm{\Lambda^{s} \uu^{R}}_{L^{2}}^{2} + \nu \norm{\Lambda^{s} \Grad \uu^{R}}_{L^{2}}^{2} &= \inner{\Lambda^{s} [(\BB^{R} \cdot \Grad) \BB^{R}]}{\Lambda^{s} \uu^{R}} \\ &\qquad - \inner{\Lambda^{s} [(\uu^{R} \cdot \Grad) \uu^{R}]}{\Lambda^{s} \uu^{R}},\\
\frac{1}{2} \frac{\rd}{\rd t} \norm{\Lambda^{s} \BB^{R}}_{L^{2}}^{2} &= \inner{ \Lambda^{s} [(\BB^{R} \cdot \Grad) \uu^{R}] }{\Lambda^{s} \BB^{R}} \\ &\qquad  - \inner{ \Lambda^{s} [(\uu^{R} \cdot \Grad) \BB^{R}] }{\Lambda^{s} \BB^{R}}.
\end{align*}
Note that we have used the fact that $\Ss_{R} \uu^{R} = \uu^{R}$, since $\uu^{R} \in V_{R}$.

The most difficult term, $\inner{ \Lambda^{s} [(\uu^{R} \cdot \Grad) \BB^{R}] }{\Lambda^{s} \BB^{R}}$, is dealt with easily by our commutator estimate (Corollary~\ref{cor:Commutator}):
\[
\abs{\inner{\Lambda^{s} [(\uu^{R} \cdot \Grad) \BB^{R}]}{\Lambda^{s} \BB^{R}}} \leq c \norm{\Grad \uu^{R}}_{H^{s}} \norm{\BB^{R}}_{H^{s}}^{2}.
\]
The other three terms can be estimated using the fact that $H^{s}$ is an algebra for $s > n/2$. Two follow directly:
\begin{align*}
\abs{\inner{\Lambda^{s} [(\uu^{R} \cdot \Grad) \uu^{R}]}{\Lambda^{s} \uu^{R}}} &\leq c \norm{\Grad \uu^{R}}_{H^{s}} \norm{\uu^{R}}_{H^{s}}^{2}, \\
\abs{\inner{\Lambda^{s} [(\BB^{R} \cdot \Grad) \uu^{R}]}{\Lambda^{s} \BB^{R}}} &\leq c \norm{\Grad \uu^{R}}_{H^{s}} \norm{\BB^{R}}_{H^{s}}^{2},
\end{align*}
while the remaining term requries an integration by parts:
\begin{align*}
\abs{\inner{\Lambda^{s} [(\BB^{R} \cdot \Grad) \BB^{R}]}{\Lambda^{s} \uu^{R}}} &= \abs{- \sum_{i,j=1}^{n} \int_{\R^{n}} \Lambda^{s} [\BB^{R}_{i} \BB^{R}_{j}] \Lambda^{s} \pd_{i} \uu^{R}_{j}} \\
&\leq c \norm{\BB^{R}}_{H^{s}}^{2} \norm{\Grad \uu^{R}}_{H^{s}}. 
\end{align*}
Hence
\[
\frac{1}{2} \frac{\rd}{\rd t} ( \norm{\uu^{R}}_{\dot{H}^{s}}^{2} + \norm{\BB^{R}}_{\dot{H}^{s}}^{2}) + \nu \norm{\Grad \uu^{R}}_{\dot{H}^{s}}^{2} \leq c \norm{\Grad \uu^{R}}_{H^{s}} (\norm{\uu^{R}}_{H^{s}}^{2} + \norm{\BB^{R}}_{H^{s}}^{2}).
\]
Combining this with the energy estimate \eqref{eqn:EnergyEquality} yields
\[
\frac{1}{2} \frac{\rd}{\rd t} ( \norm{\uu^{R}}_{H^{s}}^{2} + \norm{\BB^{R}}_{H^{s}}^{2}) + \nu \norm{\Grad \uu^{R}}_{H^{s}}^{2} \leq c \norm{\Grad \uu^{R}}_{H^{s}} (\norm{\uu^{R}}_{H^{s}}^{2} + \norm{\BB^{R}}_{H^{s}}^{2}).
\]
By Young's inequality,
\[
\frac{\rd}{\rd t} ( \norm{\uu^{R}}_{H^{s}}^{2} + \norm{\BB^{R}}_{H^{s}}^{2}) + \nu \norm{\Grad \uu^{R}}_{H^{s}}^{2} \leq \frac{c}{\nu} (\norm{\uu^{R}}_{H^{s}}^{2} + \norm{\BB^{R}}_{H^{s}}^{2})^{2}.
\]
Setting $Y(t) = ( \norm{\uu^{R}}_{H^{s}}^{2} + \norm{\BB^{R}}_{H^{s}}^{2})$ and $Y_{0} = ( \norm{\uu_{0}}_{H^{s}}^{2} + \norm{\BB_{0}}_{H^{s}}^{2})$, a standard Gronwall-type argument shows that
\begin{equation}
\label{eqn:UniformBoundY}
Y(t) \leq \frac{\nu Y_{0}}{\nu - CTY_{0}}
\end{equation}
for all $t \in [0,T]$. So provided we choose $T_{*} < CY_{0}/\nu$, $\norm{\uu^{R}}_{H^{s}}$ and $\norm{\BB^{R}}_{H^{s}}$ remain bounded on $[0,T_{*}]$ independently of $R$, and $\int_{0}^{T_{*}} \norm{\Grad \uu^{R}(t)}_{H^{s}}^{2} \, \rd t$ is bounded uniformly in $R$.
\end{proof}

Having proven these uniform estimates, we could use the compactness theorem of \cite{art:Aubin1963} and \cite{book:Lions1969} (see also \cite{art:Simon1987}) to extract a subsequence $(\uu^{R_{m}}, \BB^{R_{m}})$ that converges strongly to $(\uu, \BB)$ in some sense; while this approach is natural when working on a bounded domain, on the whole space one only obtains the requisite strong convergence on compact subsets, and one must then appeal to the argument of, for example, \cite{book:CDGG}, \S2.2.4, to show that, indeed, the nonlinear terms converge as required.

In order to avoid this, we instead follow the approach of, for example, \cite{book:MajdaBertozzi} and show that $\uu^{R}$ and $\BB^{R}$ converge strongly in $L^{\infty}(0, T_{*}; L^{2}(\R^{n}))$, by showing they are Cauchy as $R \to \infty$.

\begin{prop}
\label{prop:Cauchy}
The family $(\uu^{R}, \BB^{R})$ of solutions of \eqref{eqn:MHD-Mollified} are Cauchy (as $R \to \infty$) in $L^{\infty}(0, T; L^{2}(\R^{n}))$. 
\end{prop}

\begin{proof}
Consider again equations~\eqref{eqn:MHD-Mollified}, and take the difference between the equations for $R$ and $R'$:
\begin{subequations}
\label{eqn:MHD-Difference}
\begin{align}
\frac{\pd}{\pd t} (\uu^{R} - \uu^{R'}) - &\nu \Laplace (\uu^{R} - \uu^{R'}) + \Grad (p_{*}^{R} - p_{*}^{R'}) \notag \\
&= \Ss_{R}[(\BB^{R} \cdot \Grad) \BB^{R}] - \Ss_{R'}[(\BB^{R'} \cdot \Grad) \BB^{R'}] \notag \\
&\qquad - \Ss_{R}[(\uu^{R} \cdot \Grad) \uu^{R}] + \Ss_{R'}[(\uu^{R'} \cdot \Grad) \uu^{R'}], \label{eqn:MHD-Difference-u} \\
\frac{\pd}{\pd t} (\BB^{R} - \BB^{R'}) &= \Ss_{R}[(\BB^{R} \cdot \Grad) \uu^{R}] - \Ss_{R'}[(\BB^{R'} \cdot \Grad) \uu^{R'}] \notag \\
&\qquad - \Ss_{R}[(\uu^{R} \cdot \Grad) \BB^{R}] + \Ss_{R'}[(\uu^{R'} \cdot \Grad) \BB^{R'}]. \label{eqn:MHD-Difference-B}
\end{align}
\end{subequations}
Take the inner product of \eqref{eqn:MHD-Difference-u} with $\uu^{R} - \uu^{R'}$ and the inner product of \eqref{eqn:MHD-Difference-B} with $\BB^{R} - \BB^{R'}$ and add to obtain
\begin{subequations}
\label{eqn:MHD-IP}
\begin{align}
&\frac{1}{2} \frac{\rd}{\rd t} \left( \norm{\uu^{R} - \uu^{R'}}_{L^{2}}^{2} + \norm{\BB^{R} - \BB^{R'}}_{L^{2}}^{2} \right) + \nu \norm{\Grad(\uu^{R} - \uu^{R'})}_{L^{2}}^{2} \notag \\
&= \inner{\Ss_{R}[(\BB^{R} \cdot \Grad) \BB^{R}] - \Ss_{R'}[(\BB^{R'} \cdot \Grad) \BB^{R'}]}{\uu^{R} - \uu^{R'}} \label{eqn:BBu} \\
&\qquad - \inner{\Ss_{R}[(\uu^{R} \cdot \Grad) \uu^{R}] - \Ss_{R'}[(\uu^{R'} \cdot \Grad) \uu^{R'}]}{\uu^{R} - \uu^{R'}} \label{eqn:uuu} \\
&\qquad + \inner{\Ss_{R}[(\BB^{R} \cdot \Grad) \uu^{R}] - \Ss_{R'}[(\BB^{R'} \cdot \Grad) \uu^{R'}]}{\BB^{R} - \BB^{R'}} \label{eqn:BuB} \\
&\qquad - \inner{\Ss_{R}[(\uu^{R} \cdot \Grad) \BB^{R}] - \Ss_{R'}[(\uu^{R'} \cdot \Grad) \BB^{R'}]}{\BB^{R} - \BB^{R'}}. \label{eqn:uBB}
\end{align}
\end{subequations}

We split each term into three parts: for \eqref{eqn:uBB}, for example, we get
\begin{subequations}
\begin{align}
&\inner{\Ss_{R}[(\uu^{R} \cdot \Grad) \BB^{R}] - \Ss_{R'}[(\uu^{R'} \cdot \Grad) \BB^{R'}]}{\BB^{R} - \BB^{R'}} \notag \\
&\qquad = \inner{(\Ss_{R} - \Ss_{R'}) [(\uu^{R} \cdot \Grad) \BB^{R}]}{\BB^{R} - \BB^{R'}} \label{eqn:uBB1} \\
&\qquad + \inner{\Ss_{R'}[((\uu^{R} - \uu^{R'}) \cdot \Grad) \BB^{R}]}{\BB^{R} - \BB^{R'}} \label{eqn:uBB2} \\
&\qquad + \inner{\Ss_{R'}[(\uu^{R'} \cdot \Grad)(\BB^{R} - \BB^{R'})]}{\BB^{R} - \BB^{R'}}. \label{eqn:uBB3}
\end{align}
\end{subequations}
Notice that \eqref{eqn:uBB3} is zero (integrating by parts and using the divergence-free condition).

For \eqref{eqn:uBB1}, we use the cutoff property \eqref{eqn:MollifierProp2} (recalling that $R' > R$) to obtain
\[
\abs{\inner{(\Ss_{R} - \Ss_{R'}) [(\uu^{R} \cdot \Grad) \BB^{R}]}{\BB^{R} - \BB^{R'}}} \leq \frac{c}{R^{\eps}} \norm{\uu^{R}}_{H^{s}} \norm{\BB^{R}}_{H^{s}} \norm{\BB^{R} - \BB^{R'}}_{L^{2}}
\]
provided $0 < \eps < s-1$; the other three corresponding terms are handled in the same way.

The most difficult term is \eqref{eqn:uBB2}, which requires more care: in particular it requires different treatments in two and three dimensions: in 2D, we use
\[
\norm{fg}_{L^{2}} \leq \norm{f}_{L^{2/\eps}} \norm{g}_{L^{2/(1+\eps)}} \leq c \norm{f}_{H^{1-\eps}} \norm{g}_{H^{\eps}} \leq c \norm{f}_{H^{1}} \norm{g}_{H^{s-1}},
\]
while in 3D we use
\[
\norm{fg}_{L^{2}} \leq \norm{f}_{L^{6}} \norm{g}_{L^{3}} \leq c \norm{f}_{H^{1}} \norm{g}_{H^{1/2}} \leq c \norm{f}_{H^{1}} \norm{g}_{H^{s-1}}.
\]
In either case, we obtain
\begin{align*}
&\abs{\inner{\Ss_{R'}[((\uu^{R} - \uu^{R'}) \cdot \Grad) \BB^{R}]}{\BB^{R} - \BB^{R'}}} \\
&\qquad \leq c \norm{\uu^{R} - \uu^{R'}}_{H^{1}} \norm{\Grad \BB^{R}}_{H^{s-1}} \norm{\Ss_{R'}[\BB^{R} - \BB^{R'}]}_{L^{2}}.
\end{align*}
We now use the inequality $ab \leq \frac{1}{\nu} a^{2} + \frac{\nu}{4} b^{2}$, yielding
\[
\eqref{eqn:uBB2} \leq \frac{\nu}{4} \norm{\uu^{R} - \uu^{R'}}_{H^{1}}^{2}  + \frac{c}{\nu} \norm{\Grad \BB^{R}}_{H^{s-1}}^{2} \norm{\Ss_{R'}[\BB^{R} - \BB^{R'}]}_{L^{2}}^{2}.
\]


All other terms can be estimated similarly, with two terms from \eqref{eqn:BBu} and \eqref{eqn:BuB} adding to zero. Putting all the terms together we obtain
\begin{align*}
&\frac{\rd}{\rd t} \left( \norm{\uu^{R} - \uu^{R'}}_{L^{2}}^{2} + \norm{\BB^{R} - \BB^{R'}}_{L^{2}}^{2} \right) + \nu \norm{\Grad(\uu^{R} - \uu^{R'})}_{L^{2}}^{2} \\
&\qquad \leq \frac{1}{R^{\eps}} \left( \norm{\uu^{R}}_{H^{s}}^{2} + \norm{\BB^{R}}_{H^{s}}^{2} \right) \left( \norm{\uu^{R} - \uu^{R'}}_{L^{2}} + \norm{\BB^{R} - \BB^{R'}}_{L^{2}} \right) \\
&\qquad \qquad + c \left( \norm{\Grad \uu^{R}}_{H^{s}} + \frac{1}{\nu} \norm{\BB_{R}}_{H^{s}}^{2} \right) \left( \norm{\uu^{R} - \uu^{R'}}_{L^{2}}^{2} + \norm{\BB^{R} - \BB^{R'}}_{L^{2}}^{2} \right).
\end{align*}
Setting $Y(t) = \norm{\uu^{R} - \uu^{R'}}_{L^{2}} + \norm{\BB^{R} - \BB^{R'}}_{L^{2}}$, and using the bound
\[
\sup_{t \in [0,T_{*}]} \norm{\uu^{R}(t)}_{H^{s}}\text{, }\sup_{t \in [0,T_{*}]} \norm{\BB^{R}(t)}_{H^{s}}\text{, } \int_{0}^{T_{*}} \norm{\Grad \uu^{R}(t)}_{H^{s}}^{2} \, \rd t \leq M
\]
for all $t \in [0, T_{*}]$, we see that
\[
\frac{\rd Y}{\rd t} \leq \frac{M}{R^{\eps}} + c Y \left( \frac{M}{\nu} + \norm{\Grad \uu^{R}}_{H^{s}} \right).
\]
As $\norm{\Grad \uu^{R}}_{H^{s}}$ is integrable in time, a standard Gronwall argument shows that
\[
\sup_{t \in [0,T_{*}]} Y(t) \leq \frac{C(\nu,M,T_{*})}{R^{\eps}},
\]
and the right-hand side tends to zero as as $R, R' \to \infty$, as required.
\end{proof}

It follows that $(\uu^{R}, \BB^{R}) \to (\uu, \BB)$ strongly in $L^{\infty}(0, T_{*}; L^{2}(\R^{n}))$, and it \linebreak is straightforward to use the last estimate in the proof above to show that \linebreak $\Grad \uu^{R} \to \Grad \uu$  strongly in $L^{2}(0, T_{*}; L^{2}(\R^{n}))$. Combining Propositions~\ref{prop:Uniform} and \linebreak \ref{prop:Cauchy} and using Sobolev interpolation (see \cite{book:AdamsFournier}) yields \linebreak $(\uu^{R}, \BB^{R}) \to (\uu, \BB)$ strongly in $L^{\infty}(0, T_{*}; H^{s'}(\R^{n}))$ for any $s' < s$. Furthermore, \linebreak $\Grad \uu^{R} \to \Grad \uu$ strongly in $L^{2}(0, T_{*}; H^{s'}(\R^{n}))$ for any $s' < s$, and thus $\Laplace \uu^{R} \to \Laplace \uu$ \linebreak strongly in $L^{2}(0, T_{*}; H^{s'-1}(\R^{n}))$. 

To deal with the nonlinear terms, we prove a simple estimate.

\begin{lemma}
\label{lem:NiceEstimate}
Fix $s > n/2$ and let $\vv, \ww \in H^{s}$  with $\Div \vv = 0$. Then
\[
\norm{(\vv \cdot \Grad) \ww}_{H^{s-1}} \leq C \norm{\vv}_{H^{s}} \norm{\ww}_{H^{s}}.
\]
\end{lemma}

\begin{proof}
As $\vv$ is divergence-free, $(\vv \cdot \Grad) \ww = \Div (\vv \otimes \ww)$. As $H^{s}$ is an algebra,
\[
\norm{(\vv \cdot \Grad) \ww}_{H^{s-1}} = \norm{\Div (\vv \otimes \ww)}_{H^{s-1}} \leq C \norm{\vv \otimes \ww}_{H^{s}} \leq C \norm{\vv}_{H^{s}} \norm{\ww}_{H^{s}}. \qedhere
\]
\end{proof}

For $s' > n/2$, by Lemma~\ref{lem:NiceEstimate},
\[
\sup_{t \in [0,T_{*}]} \norm{\Ss_{R} [(\uu^{R} \cdot \Grad) \BB^{R}] - (\uu \cdot \Grad)\BB}_{H^{s'-1}} \to 0
\]
as $R \to \infty$. It remains to show convergence of the time derivatives: using Lemma~\ref{lem:NiceEstimate} once more, we obtain
\[
\bignorm{\frac{\pd \uu^{R}}{\pd t}}_{H^{s-1}} + \bignorm{\frac{\pd \BB^{R}}{\pd t}}_{H^{s-1}} \leq C \norm{\Laplace \uu^{R}}_{H^{s-1}} + C \left( \norm{\uu^{R}}_{H^{s}} + \norm{\BB^{R}}_{H^{s}} \right)^{2}.
\]
Thus, using this and Proposition~\ref{prop:Uniform}, we can extract a subsequence $R_{m} \to +\infty$ such that
\[
\frac{\pd \uu^{R_{m}}}{\pd t} \weakstarto \frac{\pd \uu}{\pd t}, \frac{\pd \BB^{R_{m}}}{\pd t} \weakstarto \frac{\pd \BB}{\pd t} \text{ in } L^{2}(0, T_{*}; H^{s-1}(\R^{n})).
\]
Using the above strong convergence allows us to conclude that the time derivatives will converge strongly in $L^{2}(0, T_{*}; H^{s'-1}(\R^{n}))$ as well, and hence $(\uu, \BB)$ solves \eqref{eqn:MHD} as an equality in $L^{2}(0, T_{*}; H^{s'-1}(\R^{n}))$. Finally, the uniform bounds in Proposition~\ref{prop:Uniform} guarantee the existence of a subsequence (which we relabel) such that
\begin{align*}
\uu^{R_{m}} \weakstarto \uu, \BB^{R_{m}} \weakstarto \BB &\text{ in } L^{\infty}(0, T_{*}; H^{s}(\R^{n})), \\
\Grad \uu^{R_{m}} \weakstarto \Grad \uu &\text{ in }L^{2}(0, T_{*}; H^{s}(\R^{n}))
\end{align*}
(by the Banach--Alaoglu theorem), which guarantees that the limit satisfies
\[
\uu \in L^{\infty}(0, T_{*}; H^{s}(\R^{n})) \cap L^{2}(0, T_{*}; H^{s+1}(\R^{n})), \qquad \BB \in L^{\infty}(0, T_{*}; H^{s}(\R^{n})).
\]

As $\uu \in L^{2}(0, T_{*}; H^{s+1}(\R^{n}))$ and $\frac{\pd \uu}{\pd t} \in L^{2}(0, T_{*}; H^{s-1}(\R^{n}))$, by standard results (see, e.g., \cite{book:Evans}, \S5.9, Theorem 4), $\uu \in C([0, T_{*}]; H^{s}(\R^{n}))$. However, a further argument is needed to show that $\BB \in C([0, T_{*}]; H^{s}(\R^{n}))$: we proceed as in Theorem 3.5 (pp109--111) in \cite{book:MajdaBertozzi}, without going into the details, using the argument used for the Euler equations. It is easy to show, using the bounds in Proposition~\ref{prop:Uniform}, that $\BB \in C_{\mathrm{W}}([0, T_{*}]; H^{s}(\R^{n}))$; that is, $\BB$ is continuous in the weak topology of $H^{s}$. It thus suffices to show that $\norm{\BB(\cdot)}_{H^{s}}$ is continuous as a function of time. For fixed $\uu$ such that $\Grad \uu \in L^{2}(0, T_{*}; H^{s}(\R^{n}))$, proceeding analogously to Proposition~\ref{prop:Cauchy}, we obtain
\[                                                                                                                                                                                                                                                                                                                                                                 
\frac{1}{2} \frac{\rd}{\rd t} \norm{\BB(t)}_{H^{s}} \leq c \norm{\Grad \uu}_{H^{s}} \norm{\BB}_{H^{s}}^{2},
\]
and Gronwall's inequality shows that
\[
\norm{\BB(t)}_{H^{s}} \leq \norm{\BB_{0}}_{H^{s}} \exp \left( \int_{0}^{t} \norm{\Grad \uu(\tau)}_{H^{s}} \, \rd \tau \right),
\]
and hence $\norm{\BB(\cdot)}_{H^{s}}$ is continuous from the right at time $t=0$; applying this bound to the equation started at an arbitrary time $\tau \in [0, T_{*}]$ shows that $\norm{\BB(\cdot)}_{H^{s}}$ is continuous from the right at time $t=\tau$. But the $\BB$ equation is time-reversible, so $\norm{\BB(\cdot)}_{H^{s}}$ is continuous from the left at time $t=\tau$, and as $\tau$ was arbitrary $\norm{\BB(\cdot)}_{H^{s}}$ is continuous. This, combined with the fact that $\BB \in C_{\mathrm{W}}([0, T_{*}]; H^{s}(\R^{n}))$, yields that $\BB \in C([0, T_{*}]; H^{s}(\R^{n}))$).

The proof of uniqueness is very similar to the proof of Proposition~\ref{prop:Cauchy}, and we omit it. This completes the proof of Theorem~\ref{thm:MHDLocalExistence}.

\section{Local existence for a reduced model}
\label{sec:LocalExistenceCEP}

Consider now the equations
\begin{subequations}
\label{eqn:StokesMHD}
\begin{align}
- \nu \Laplace \uu + \Grad p_{*} &= (\BB \cdot \Grad) \BB, \label{eqn:StokesMHD-u} \\
\frac{\pd \BB}{\pd t} + (\uu \cdot \Grad) \BB &= (\BB \cdot \Grad) \uu, \label{eqn:StokesMHD-B} \\
\Div \uu = \Div \BB &= 0
\end{align}
\end{subequations}
on the whole of $\R^{n}$, with divergence-free initial data $\BB_{0} \in H^{s}(\R^{n})$, for $s > n/2$.

Our interest in this reduced model stems from the method of magnetic relaxation described in the introduction. If all we are interested in is the limiting state, the dynamical model used to obtain that steady state is not particularly important: in a talk given at the University of Warwick, \cite{misc:talkMoffatt2009} argued that dropping the acceleration terms from the $\uu$ equation and working with a ``Stokes'' model --- such as equations~\eqref{eqn:StokesMHD} --- might prove more mathematically amenable.

In \cite{art:ARMA}, global existence and uniqueness of solutions in 2D, and local existence of solutions in 3D, was established for a variant of \eqref{eqn:StokesMHD} with a diffusion term $-\eta \Laplace \BB$ in the second equation. In this section, we establish local existence and uniqueness of solutions for \eqref{eqn:StokesMHD} (without magnetic diffusion) in $H^{s}$ for $s > n/2$.

\begin{thm}
\label{thm:StokesMHDLocalExistence}
For $s > n/2$, and initial data $\BB_{0} \in H^{s}(\R^{n})$ with $\Div \BB_{0} = 0$, there exists a time $T_{*} = T_{*}(s, \norm{\BB_{0}}_{H^{s}}) > 0$ such that the equations \eqref{eqn:StokesMHD} have a unique solution $(\uu, \BB)$, such that $\BB \in C([0, T_{*}]; H^{s}(\R^{n}))$ and $\uu \in C([0, T_{*}]; H^{s+1}(\R^{n}))$.
\end{thm}

In this case, we consider the truncated equations:
\begin{subequations}
\label{eqn:StokesMHD-Mollified}
\begin{align}
- \nu \Laplace \uu^{R} +  \Grad p_{*}^{R} &= (\BB^{R} \cdot \Grad) \BB^{R}, \label{eqn:StokesMHD-Mollified-u} \\
\frac{\pd \BB^{R}}{\pd t} &= \Ss_{R}[(\BB^{R} \cdot \Grad) \uu^{R}] - \Ss_{R}[(\uu^{R} \cdot \Grad) \BB^{R}], \label{eqn:StokesMHD-Mollified-B} \\
\Div \uu^{R} &= \Div \BB^{R} = 0,
\end{align}
\end{subequations}
with initial data $\BB_{0} \in H^{s}(\R^{n})$. Using standard elliptic regularity results in conjunction with Lemma~\ref{lem:NiceEstimate}, we see that
\begin{align}
\norm{\uu^{R} - \uu^{R'}}_{H^{s+1}} &\leq \frac{1}{\nu} \norm{(\BB^{R} \cdot \Grad) \BB^{R} - (\BB^{R'} \cdot \Grad) \BB^{R'}}_{H^{s-1}} \notag \\
&\leq \frac{1}{\nu} \left( \norm{\BB^{R} - \BB^{R'}}_{H^{s}} \norm{\BB^{R}}_{H^{s}} + \norm{\BB^{R'}}_{H^{s}} \norm{\BB^{R} - \BB^{R'}}_{H^{s}} \right), \label{eqn:EllipticEstimate}
\end{align}
so on $V_{R}$, $\uu^{R}$ is a Lipschitz function of $\BB^{R}$. Thus, as before, the second equation (for $\BB$) is a Lipschitz ODE on the space $V_{R}$, and by Picard's theorem has a solution for as long as $\norm{\BB^{R}}_{H^{s}}$ remains finite.

By the same techniques as Proposition~\ref{prop:Uniform}, we obtain the uniform bound
\[
\frac{1}{2} \frac{\rd}{\rd t} \norm{\BB^{R}}_{H^{s}}^{2} + \nu \norm{\Grad \uu^{R}}_{H^{s}}^{2} \leq c \norm{\Grad \uu^{R}}_{H^{s}} \norm{\BB^{R}}_{H^{s}}^{2},
\]
and a Gronwall argument again shows there is some short time $T_{*}$ such that $\BB^{R}$ are uniformly bounded in $L^{\infty}(0, T_{*}; H^{s}(\R^{n}))$. Furthermore, using Lemma~\ref{lem:NiceEstimate}, we obtain
\[
\norm{\uu}_{H^{s+1}} \leq \norm{(\BB \cdot \Grad) \BB}_{H^{s-1}} \leq \norm{\BB}_{H^{s}}^{2},
\]
so $\uu^{R}$ are uniformly bounded in $L^{\infty}(0, T_{*}; H^{s+1}(\R^{n}))$.

An almost identical argument to Proposition~\ref{prop:Cauchy} --- which we omit here --- shows that $\BB^{R} \to \BB$ strongly in $L^{\infty}(0, T_{*}; L^{2}(\R^{n}))$ and $\Grad \uu^{R} \to \Grad \uu$ strongly in $L^{2}(0, T_{*}; L^{2}(\R^{n}))$. Interpolation thus yields that, for any $s' < s$, $\BB^{R} \to \BB$ strongly in $L^{\infty}(0, T_{*}; H^{s'}(\R^{n}))$, and $\uu^{R} \to \uu$ strongly in $L^{\infty}(0, T_{*}; H^{s'+1}(\R^{n}))$. Hence $\Laplace \uu^{R} \to \Laplace \uu$ strongly in $L^{\infty}(0, T_{*}; H^{s'-1}(\R^{n}))$.

Convergence of the nonlinear terms is handled in the same way as the previous case, and thus $(\uu, \BB)$ solves \eqref{eqn:StokesMHD} as an equality in $H^{s'-1}$. Again, the Banach--Alaoglu theorem guarantees that the limit $\uu \in L^{\infty}(0, T_{*}; H^{s+1}(\R^{n}))$ and $\BB \in L^{\infty}(0, T_{*}; H^{s}(\R^{n}))$. Exactly the same argument as the previous case applies to show that in fact $\BB \in C([0, T_{*}]; H^{s}(\R^{n}))$; thence, an argument analogous to \eqref{eqn:EllipticEstimate} for $\uu(t_{1}) - \uu(t_{2})$ shows that $\uu \in C([0, T_{*}]; H^{s+1}(\R^{n}))$.

Finally, uniqueness is handled similarly to the previous case, and thus the proof of Theorem~\ref{thm:StokesMHDLocalExistence} is complete.

\section{Conclusion}

In Theorem~\ref{thm:MHDLocalExistence}, we established local existence and uniqueness of solutions to \eqref{eqn:MHD} in $H^{s}$ for $s > n/2$. A natural question to ask is whether this can be generalised to $H^{n/2}$: the counterexample to inequality \eqref{eqn:Commutator} outlined in Appendix~\ref{app:Counterexample} shows that the same approach will not work. It may prove fruitful to consider local existence in Besov or Triebel--Lizorkin spaces with the same scaling as $H^{n/2}$, which we hope to examine in a future paper (cf. the result of \cite{art:Chae2003} for the Euler equations in critical Triebel--Lizorkin spaces).

\appendix
\section{A counterexample to Theorem~\ref{thm:Commutator} in $H^{1}(\R^{2})$}
\label{app:Counterexample}

In this appendix, we show that Theorem~\ref{thm:Commutator} cannot hold for $s = n/2$, at least in two dimensions (when $s = 1$). More precisely, we show that the inequality 
\begin{equation}
\norm{\pd_{k} [ (\uu \cdot \Grad) \BB] - (\uu \cdot \Grad) (\pd_{k} \BB)}_{L^{2}} \leq c \norm{\Grad \uu}_{H^{1}} \norm{\BB}_{H^{1}} \tag{\ref{eqn:False2DIneq}}
\end{equation}
cannot hold in dimension $2$, by exhibiting a pair of divergence-free functions $\uu$ and $\BB$ for which the right-hand side is finite, but the left-hand side is infinite.

Since we have one full derivative, we can make an important simplification by means of the product rule: the inequality reduces to
\begin{equation}
\label{eqn:WantCounterexample}
\norm{((\pd_{k} \uu) \cdot \Grad) \BB}_{L^{2}} \leq c \norm{\Grad \uu}_{H^{1}} \norm{\BB}_{H^{1}}.
\end{equation}
Now, Theorem~\ref{thm:Commutator} does not require $\uu$ and $\BB$ to be divergence-free, and there is an easier counterexample to \eqref{eqn:WantCounterexample} if we drop the divergence-free requirement. However, in order to eliminate the possibility that \eqref{eqn:WantCounterexample} might hold for divergence-free vector fields, even if it does not hold in general, we present here a counterexample in which $\uu$ and $\BB$ are divergence-free.

Since we are in two dimensions, we may represent our divergence-free vector fields as $\uu = \Grad^{\perp} \phi$ and $\BB = \Grad^{\perp} \psi$ for some scalar functions $\phi$ and $\psi$; in other words,
\[
\uu = (\pd_{2} \phi, -\pd_{1} \phi), \qquad \BB = (\pd_{2} \psi, -\pd_{1} \psi).
\]
Thus
\[
((\pd_{k} \uu) \cdot \Grad) \BB_{1} = (\pd_{k} \uu_{1}) (\pd_{1} \BB_{1}) + (\pd_{k} \uu_{2}) (\pd_{2} \BB_{1})
\]
becomes
\[
((\pd_{k} \uu) \cdot \Grad) \BB_{1} = (\pd_{k} \pd_{2} \phi) (\pd_{1} \pd_{2} \psi) - (\pd_{k} \pd_{1} \phi) (\pd_{2}^{2} \psi)
\]
(one can treat the second component similarly). Taking Fourier transforms of both sides yields
\begin{equation}
\label{eqn:FTtoBound}
\mathscr{F}[((\pd_{k} \uu) \cdot \Grad) \BB_{1}](\xi) = 16 \pi^{4} \int \underbrace{\zeta_{k} (\xi - \zeta)_{2} [ \zeta^{\perp} \cdot (\xi - \zeta) ]}_{(*)} \hat{\phi}(\zeta) \hat{\psi}(\xi - \zeta) \, \rd \zeta,
\end{equation}
By choosing the support of $\hat{\phi}$ and $\hat{\psi}$ to lie in certain small sectors, we may bound the expression $(*)$ below by the absolute values of the respective components; that is,
\[
\zeta_{k} (\xi - \zeta)_{2} [ \zeta^{\perp} \cdot (\xi - \zeta) ] \geq M_{\delta} |\zeta|^{2} |\xi - \zeta|^{2}.
\]
This is made precise in the following lemma. (The proof thereof is largely elementary, using the bound $\sin x \geq 1 - \frac{2}{\pi} |x - \frac{\pi}{2}|$ for $x \in (0,\pi)$, and we omit the details.)

\begin{lemma}
\label{lem:LowerBound}
Fix $0 < \delta < \frac{1}{\sqrt{2}}$. Suppose that $\zeta, \eta \in \R^{2}$ satisfy $|\arg \zeta - \frac{\pi}{4}| < \delta$, $|\arg \eta - \frac{3\pi}{4}| < \delta$. Then
\[
\frac{\zeta_{k}}{|\zeta|} \frac{\eta_{2}}{|\eta|} \frac{[ \zeta^{\perp} \cdot \eta ]}{|\zeta| |\eta|} \geq \left( \tfrac{\sqrt{2}}{2} - \delta \right)^{2} \left( 1 - \tfrac{4\delta}{\pi} \right) =: M_{\delta} > 0.
\]
\end{lemma}

\begin{subequations}
\label{eqn:phipsi}
\begin{align}
\hat{\phi} (\zeta) &= \frac{1}{|\zeta|^{2} (1 + |\zeta|^{2})^{1/2}} g(|\zeta|) h_{1}(\arg \zeta), \label{eqn:phi} \\
\hat{\psi} (\eta) &= \frac{1}{|\eta| (1 + |\eta|^{2})^{1/2}} g(|\eta|) h_{2}(\arg \eta), \label{eqn:psi}
\end{align}
\end{subequations}
where
\[
g(r) = \begin{cases} \frac{1}{r (\log r)^{\alpha}} & \text{for }r > \ee \\ 0 & \text{otherwise,} \end{cases}
\]
$\alpha > 0$ will be chosen later, and
\begin{align*}
h_{1}(\theta) &= \begin{cases} 1 & \text{for }\theta \in [\tfrac{\pi}{4} - \delta, \tfrac{\pi}{4} + \delta] \\ 0 & \text{for }\theta \notin [\tfrac{\pi}{4} - \delta, \tfrac{\pi}{4} + \delta] \end{cases} \intertext{and}
h_{2}(\theta) &= \begin{cases} 1 & \text{for }\theta \in [\tfrac{3\pi}{4} - \delta, \tfrac{3\pi}{4} + \delta] \\ 0 & \text{for }\theta \notin [\tfrac{3\pi}{4} - \delta, \tfrac{3\pi}{4} + \delta]  \end{cases}
\end{align*}
Notice that
\begin{align*}
\norm{\Grad \uu}_{H^{1}}^{2} = \norm{\Grad (\Grad^{\perp} \phi)}_{H^{1}}^{2} &= \norm{(1 + |\zeta|^{2})^{1/2} |\zeta|^{2} \hat{\phi}(\zeta)}_{L^{2}}^{2} = \int |g(|\zeta|) h_{1}(\arg \zeta)|^{2} \, \rd \zeta, \\
\norm{\BB}_{H^{1}}^{2} = \norm{(\Grad^{\perp} \psi)}_{H^{1}}^{2} &= \norm{(1 + |\eta|^{2})^{1/2} |\eta| \hat{\psi}(\eta)}_{L^{2}}^{2} = \int |g(|\eta|) h_{2}(\arg \eta)|^{2} \, \rd \eta,
\end{align*}
and hence
\[
\norm{\Grad \uu}_{H^{1}}^{2} = \norm{\BB}_{H^{1}}^{2} = 2\delta \int_{\ee}^{\infty} \frac{1}{r (\log r)^{2\alpha}} \, \rd r = \frac{2\delta}{1 - 2\alpha} (\log r)^{1-2\alpha} \bigg|_{\ee}^{\infty}
\]
which is finite iff $\alpha > 1/2$.

However, by choosing $\xi$ and $\zeta$ carefully --- which we do in full detail shortly --- we may bound the expression~\eqref{eqn:FTtoBound} below by
\[
\mathscr{F}[((\pd_{k} \uu) \cdot \Grad) \BB_{1}](\xi) \geq c \int_{\Omega} \frac{1}{|\zeta|} g(|\zeta|) g(|\xi - \zeta|) \, \rd \zeta
\]
for some sector $\Omega$ in Fourier space. For small $\zeta$, $g(|\xi - \zeta|) \approx g(|\xi|)$, so
\begin{align*}
\mathscr{F}[((\pd_{k} \uu) \cdot \Grad) \BB_{1}](\xi) &\gtrsim c g(|\xi|) \int_{\Omega} \frac{1}{|\zeta|} g(|\zeta|) \, \rd \zeta \\
&\approx \frac{c}{|\xi| (\log |\xi|)^{\alpha}} \int_{1}^{|\xi|} \frac{1}{|r| (\log r)^{\alpha}} \, \rd r \\
&= \frac{c}{|\xi| (\log |\xi|)^{2\alpha - 1}},
\end{align*}
and the right-hand side is in $L^{2}$ if and only if $\alpha > 3/4$. Hence choosing $1/2 < \alpha < 3/4$ will yield our counterexample.

To make this fully rigorous, we carefully choose at which $\xi$ we evaluate \eqref{eqn:FTtoBound}, to ensure that both $\zeta$ and $\xi - \zeta$ fall into the ranges required in Lemma~\ref{lem:LowerBound}, and thus find a lower bound for \eqref{eqn:FTtoBound}. This is the content of the following lemma.

\begin{lemma}
\label{lem:Sectors}
Let
\begin{align*}
\Xi &:= \{ \xi \in \R^{2} : \arg \xi \in [\tfrac{3\pi}{4} - \tfrac{\delta}{2}, \tfrac{3\pi}{4} + \tfrac{\delta}{2}] \},\\
\Upsilon_{\xi} &:= \{ \zeta \in \R^{2} : |\zeta| < |\xi| \sin \tfrac{\delta}{2} \text{ and } \arg \zeta \in [\tfrac{\pi}{4} - \delta, \tfrac{\pi}{4} + \delta] \}.
\end{align*}
Then
\[
\xi \in \Xi \text{, } \zeta \in \Upsilon_{\xi} \implies \arg (\xi - \zeta) \in [\tfrac{3\pi}{4} - \delta, \tfrac{3\pi}{4} + \delta].
\]
\end{lemma}

The situation is illustrated in figure~\ref{fig:Sectors}: the light shaded region is $\Xi$, while the darker shaded region is $\Upsilon_{\xi}$. We postpone the proof of the lemma to the end of the appendix.

\begin{figure}[bt]
\centering
\includegraphics{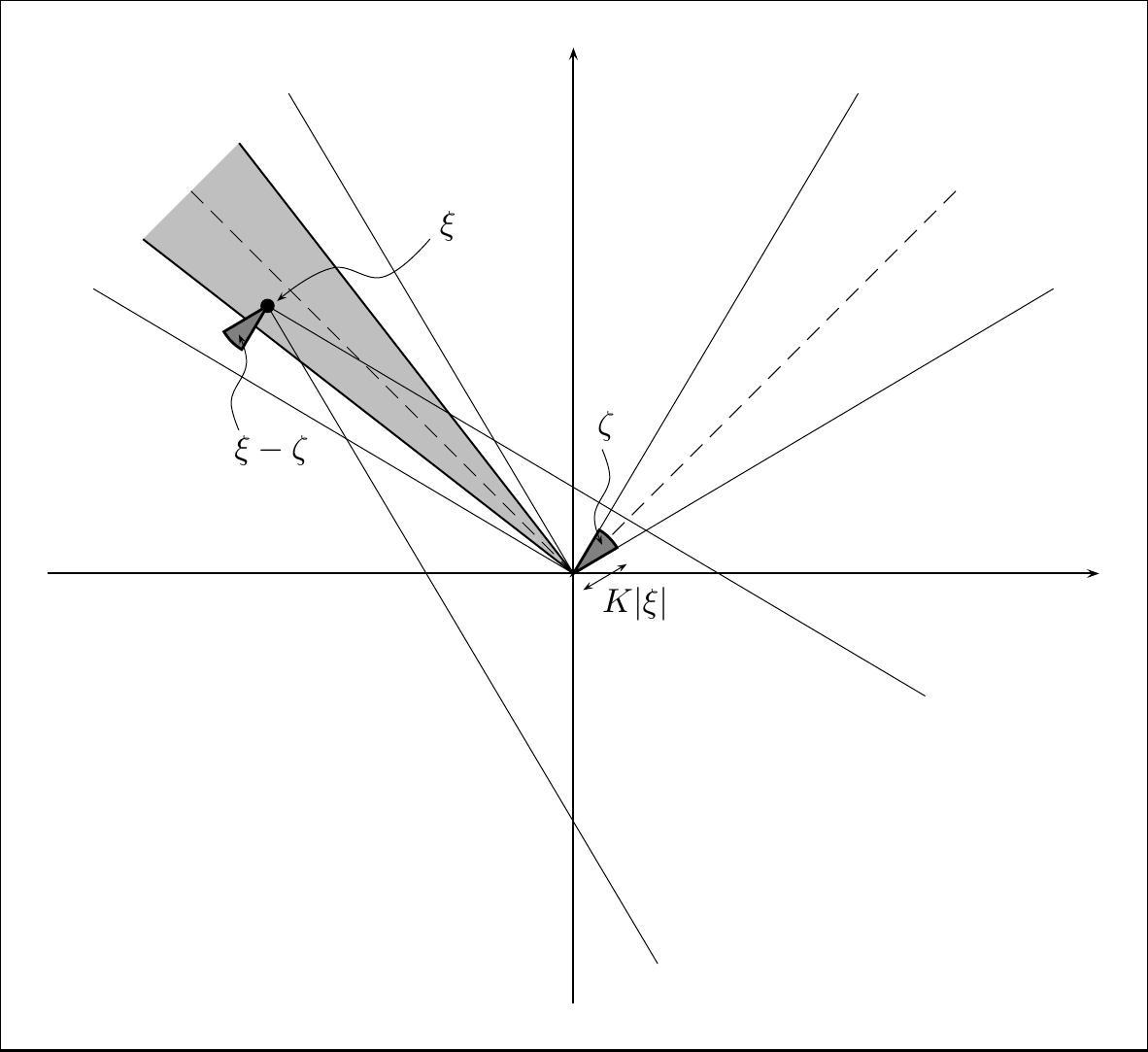}
\caption{A plot showing the sectors (in Fourier space) in which we need $\xi$ and $\zeta$ to lie, where $K = \sin \tfrac{\delta}{2}$.}
\label{fig:Sectors}
\end{figure}

We now restrict the sectors $\Xi$ and $\Upsilon_{\xi}$ to particular radii: setting $K = \sin \tfrac{\delta}{2}$, we let
\begin{align*}
X &:= \{ \xi \in \R^{2} : |\xi| > \ee / K \text{ and } \arg \xi \in [\tfrac{3\pi}{4} - \tfrac{\delta}{2}, \tfrac{3\pi}{4} + \tfrac{\delta}{2}] \} \subset \Xi,\\
Z_{\xi} &:= \{ \zeta \in \R^{2} : \ee < |\zeta| < K |\xi|  \text{ and } \arg \zeta \in [\tfrac{\pi}{4} - \delta, \tfrac{\pi}{4} + \delta] \} \subset \Upsilon_{\xi}.
\end{align*}
By staying away from the origin, we ensure that $|\xi|$ and $(1 + |\xi|^{2})^{1/2}$ are comparable: indeed, note that
\begin{equation}
\label{eqn:Comparable}
\frac{|\xi|}{(1 + |\xi|^{2})^{1/2}} \geq \frac{1}{\sqrt{2}} \qquad \text{for } |\xi| \geq 1.
\end{equation}
Hence, for $\xi \in X$, using Lemmas~\ref{lem:LowerBound} and \ref{lem:Sectors} and estimate \eqref{eqn:Comparable}, equation~\eqref{eqn:FTtoBound} reduces to
\[
\mathscr{F}[((\pd_{k} \uu) \cdot \Grad) \BB_{1}](\xi) \geq c \int_{Z_{\xi}} \frac{1}{|\zeta|} g(|\zeta|) g(|\xi - \zeta|) \, \rd \zeta
\]
for $c = 8\pi^{4}M_{\delta}$.

When $|\zeta| < K |\xi|$, we get $(1 - K)|\xi|  < |\xi - \zeta| < (1 + K) |\xi|$; and as $\delta < \pi/3$, $(1 - K) > K$, ensuring that $g((1-K)|\xi|) > 0$. Thus for $\xi \in X$ and $\zeta \in Z_{\xi}$,
\[
g(|\xi - \zeta|) \geq g((1 + K) |\xi|) > 0.
\]
Thus
\begin{align*}
\mathscr{F}[((\pd_{k} \uu) \cdot \Grad) \BB_{1}](\xi)
&\geq c g((1 + K) |\xi|) \int_{Z_{\xi}} \frac{1}{|\zeta|} g(|\zeta|) \, \rd \zeta \\
&= 2\delta c g((1 + K) |\xi|) \int_{\ee}^{K|\xi|} \frac{1}{r (\log r)^{\alpha}} \, \rd r.
\end{align*}
Since
\[
\int_{\ee}^{K|\xi|} \frac{1}{r (\log r)^{\alpha}} \, \rd r = (\log r)^{1-\alpha} \bigg|_{\ee}^{K|\xi|} = (\log K|\xi|)^{1-\alpha} - 1,
\]
we obtain
\[
\mathscr{F}[((\pd_{k} \uu) \cdot \Grad) \BB_{1}](\xi) \geq 2\delta c \frac{(\log K|\xi|)^{1-\alpha} - 1}{(1 + K) |\xi| (\log ((1 + K) |\xi|))^{\alpha}}
\]
for $\xi \in X$. We want to ensure that the left-hand side is not in $L^{2}$, so it suffices to show that the right-hand side is not square-integrable. Elementary integration yields
\begin{align*}
\norm{\mathscr{F}[((\pd_{k} \uu) \cdot \Grad) \BB_{1}]}_{L^{2}}^{2} &\geq 
c \int_{L}^{\infty} w^{2-4\alpha} \, \rd w,
\end{align*}
where $L \geq \max \{ \log \ee/K, \log (1 + K) \}$ is chosen sufficiently large such that for all $w > L$, $w^{1-\alpha} - 1 \geq \tfrac{1}{2} w^{1-\alpha}$. The last integral is finite if and only if $3 - 4\alpha < 0$, i.e.~iff $\alpha > 3/4$. Hence, choosing $1/2 < \alpha < 3/4$ ensures that $\Grad \uu \in H^{1}$ and $\BB \in H^{1}$, but that $\mathscr{F}[((\pd_{k} \uu) \cdot \Grad) \BB_{1}] \notin L^{2}$, and thus that $((\pd_{k} \uu) \cdot \Grad) \BB_{1} \notin L^{2}$.

To complete the counterexample, it only remains to prove Lemma \ref{lem:Sectors}.

\begin{proof}[Proof of Lemma~\ref{lem:Sectors}]
First, set
\begin{align*}
S_{1} &:= \{ \zeta \in \R^{2} : \arg \zeta \in [\tfrac{\pi}{4} - \delta, \tfrac{\pi}{4} + \delta] \},\\
S_{2} &:= \{ \eta \in \R^{2} : \arg \eta \in [\tfrac{3\pi}{4} - \delta, \tfrac{3\pi}{4} + \delta] \},
\end{align*}
and let $S_{3} = \xi - S_{2}$. Given $\xi \in \Xi$, we seek $\zeta$ such that $\zeta \in S_{1}$ and $\xi - \zeta \in S_{2}$: to do so, we find the largest $K(\xi)$ such that
\[
\{ \zeta \in \R^{2} : |\zeta| < K(\xi)  \text{ and } \arg \zeta \in [\tfrac{\pi}{4} - \delta, \tfrac{\pi}{4} + \delta] \} \subset S_{1} \cap S_{3}.
\]
As $\Xi \subset S_{2}$, $S_{3}$ includes zero, and is bounded by the two lines
\[
\gamma_{1}(t) = \xi + t\eta_{1}, \qquad \gamma_{2}(t) = \xi + t\eta_{2},
\]
for $t \geq 0$, where $\eta_{1} = -(\cos (\tfrac{3\pi}{4} + \delta), \sin (\tfrac{3\pi}{4} + \delta))$, $\eta_{2} = -(\cos (\tfrac{3\pi}{4} - \delta), \sin(\tfrac{3\pi}{4} - \delta))$. The line $\gamma_{2}$ has no intersection with $S_{1}$, but the line $\gamma_{1}$ will.

It thus suffices to take $K(\xi)$ to be the minimum distance of $\gamma_{1}$ to the origin: let $\xi = r(\cos (\tfrac{3\pi}{4} + s), \sin (\tfrac{3\pi}{4} + s))$. Then elementary trigonometry shows that
\[
|\gamma_{1}(t)|^{2} = |\xi + t\eta_{1}|^{2} = r^{2} + t^{2} - 2rt \cos (\delta - s).
\]
Differentiating this with respect to $t$, we see that $|\gamma_{1}(t)|^{2}$ is minimised when $t = r \cos (\delta - s)$, whence
\[
|\gamma_{1}(t)|^{2} \geq r^{2}(1 - \cos^{2} (\delta - s)) = r^{2} \sin^{2} (\delta - s).
\]
Since $s \in [-\tfrac{\delta}{2}, \tfrac{\delta}{2}]$, $\delta - s \in [\tfrac{\delta}{2}, \tfrac{3\delta}{2}]$. Hence $|\gamma_{1}(t)| \geq |\xi| \sin \tfrac{\delta}{2}$, meaning that $\Upsilon_{\xi} \subset S_{1} \cap S_{3}$, so choosing $\xi \in \Xi$ and $\zeta \in \Upsilon_{\xi}$ guarantees that $\xi - \zeta \in S_{2}$, as required.
\end{proof}

\addcontentsline{toc}{section}{Bibliography}

\bibliographystyle{agsm}
\bibliography{CommutatorsMHDPaper}

\end{document}